\def\draftdate{\today}

\documentclass{amsart}
\usepackage{stmaryrd}
\usepackage{xy}
\usepackage{mathrsfs}
\usepackage{amscd,amssymb,color}
\usepackage[colorlinks=true]{hyperref}

\numberwithin{equation}{section}

\newtheorem{theorem}[equation]{Theorem}
\newtheorem*{theorem*}{Theorem}
\newtheorem*{theoremA}{Theorem A}
\newtheorem*{theoremB}{Theorem B}

\newtheorem{proposition}[equation]{Proposition}
\theoremstyle{definition}

\theoremstyle{remark}

\DeclareMathOperator{\Mot}{Mot}

\DeclareMathOperator{\id}{id}

\newcommand{\too}{\longrightarrow}
\newcommand{\dg}{\mathsf{dg}}

\newcommand{\cA}{{\mathcal A}}
\newcommand{\cB}{{\mathcal B}}
\newcommand{\cC}{{\mathcal C}}
\newcommand{\cD}{{\mathcal D}}

\newcommand{\cM}{{\mathcal M}}

\newcommand{\cO}{{\mathcal O}}

\newcommand{\cT}{{\mathcal T}}
\newcommand{\cU}{{\mathcal U}}

\newcommand{\bbD}{\mathbb{D}}

\newcommand{\bbL}{\mathbb{L}}

\newcommand{\bbK}{I\mspace{-6.mu}K}

\newcommand{\bbN}{\mathbb{N}}
\newcommand{\bbZ}{\mathbb{Z}}

\newcommand{\ie}{\textsl{i.e.}\ }

\newcommand{\Hmo}{\mathsf{Hmo}}

\newcommand{\perf}{\mathsf{perf}} 

\newcommand{\loc}{\mathsf{loc}}
\newcommand{\Hom}{\mathsf{Hom}} 
\newcommand{\rep}{\mathsf{rep}} 
\newcommand{\dgcat}{\mathsf{dgcat}}
\newcommand{\HO}{\mathsf{HO}} 

\newcommand{\Mloc}{\Mot^{\mathsf{loc}}_{\dg}}
\newcommand{\Uloc}{\cU^{\mathsf{loc}}_{\dg}}

\newcommand{\uHom}{\underline{\mathsf{Hom}}}
\newcommand{\HomC}{\uHom_{!}}
\newcommand{\HomL}{\uHom_{\mathsf{loc}}}

\newcommand{\internalcomment}[1]{}
\xyoption{arrow}
\xyoption{matrix}
\xyoption{cmtip}
\SelectTips{cm}{}

\newdir{ >}{{}*!/-5pt/\dir{>}}

\begin{document}

\title[Bivariant cohomology, Connes' pairings, and non-commutative motives]{Bivariant cyclic cohomology and\\ Connes' bilinear pairings in \\non-commutative motives}
\author{Gon{\c c}alo~Tabuada}
\address{Departamento de Matem{\'a}tica e CMA, FCT-UNL, Quinta da Torre, 2829-516 Caparica,~Portugal}
\email{tabuada@fct.unl.pt}

\date{\draftdate}


\begin{abstract}
In this article we further the study of non-commutative motives, initiated in \cite{CT,CT1,Duke}. We prove that bivariant cyclic cohomology (and its variants) becomes representable in the category $\Mloc(e)$ of non-commutative motives. Furthermore, Connes' bilinear pairings correspond to the composition operation in $\Mloc(e)$. As an application, we obtain a simple model, given in terms of infinite matrices, for the (de)suspension of these bivariant cohomology theories.
\end{abstract}
\maketitle
\section{Introduction and statement of results}
\subsection{Non-commutative motives}
A {\em differential graded (=dg) category}, over a commutative base ring $k$, is a category enriched over 
complexes of $k$-modules (morphisms sets are complexes)
in such a way that composition fulfills the Leibniz rule\,:
$d(f\circ g)=(df)\circ g+(-1)^{\textrm{deg}(f)}f\circ(dg)$.
Dg categories enhance and solve many of the technical problems inherent to triangulated categories;
see Keller's ICM adress~\cite{ICM}. In {\em non-commutative algebraic geometry} in the sense of Drinfeld, Kaledin, Kontsevich, and others \cite{Drinfeld,Drinfeld1,Kaledin,ENS,IAS,finMot}, dg categories are considered as dg-enhancements of bounded derived categories of (quasi-)coherent sheaves on a
hypothetic non-commutative space.

All the classical invariants such as cyclic homology (and its variants), algebraic $K$-theory, and even topological cyclic homology, extend naturally from $k$-algebras to dg categories.
In order to study all these invariants simultaneously the author
introduced in \cite{Duke} the notion of {\em localizing invariant}. This notion, makes use of the language of Grothendieck derivators (see \S\ref{sec:Grothendieck}), a formalism which
allows us to state and prove precise universal properties.
Let $\mathit{L}: \HO(\dgcat) \to \bbD$ be a morphism of derivators, from the
derivator associated to the derived Morita model structure on dg categories (see \S\ref{sub:dg}), to
a triangulated derivator.
We say that $\mathit{L}$ is a localizing invariant if it preserves filtered homotopy and sends exact sequences of dg categories
\begin{eqnarray*}
\cA \too \cB \too \cC & \mapsto & \mathit{L}(\cA) \too \mathit{L}(\cB) \too \mathit{L}(\cC) \too \mathit{L}(\cA)[1]
\end{eqnarray*}
to distinguished triangles in the base category $\bbD(e)$ of $\bbD$.
Thanks to the work of Keller~\cite{Exact}, Thomason-Trobaugh~\cite{Thomason}, and Blumberg-Mandell~\cite{BM} (see also \cite{AGT}) all the mentioned invariants\footnote{In the case of algebraic
$K$-theory we consider its non-connective version.} give rise to localizing invariants. 
In \cite{Duke} the {\em universal localizing invariant} was constructed
$$ \Uloc: \HO(\dgcat) \too \Mloc\,.$$
Given any triangulated derivator $\bbD$ we have an induced equivalence of categories
\begin{equation}\label{eq:cat}
(\Uloc)^{\ast}: \HomC(\Mloc, \bbD) \stackrel{\sim}{\too} \HomL(\HO(\dgcat), \bbD)\,,
\end{equation}
where the left-hand side denotes the category of homotopy colimit preserving morphisms of derivators,
and the right-hand side denotes the category of localizing invariants.
Because of this universality property, which is a reminiscence of motives, $\Mloc$ is called the {\em localizing motivator}, and its base (triangulated) category $\Mloc(e)$ the category of {\em non-commutative motives}. We invite the reader to consult \cite{BT,CT,CT1,Duke,Chern} for applications of this theory of non-commutative motives to the Farrell-Jones isomorphism conjectures, to Kontsevich's non-commutative mixed motives, to the construction of higher Chern characters, etc.

A fundamental problem in the theory of non-commutative motives is the computation of morphisms in $\Mloc(e)$ and the description of its composition operation. In \cite{CT,CT1,Duke} an important step towards the solution of this problem was taken: let $\cA$ be a {\em saturated} dg category in the sense of Kontsevich~\cite{ENS,IAS}, \ie its complexes of morphisms are perfect and $\cA$ is perfect as a bimodule over itself. Dg categories of perfect complexes associated to smooth and proper schemes and algebras of finite global cohomological dimension are classical examples. Then, for any $n \in \bbZ$ and dg category $\cB$ we have a natural isomorphism in $\Mloc(e)$
\begin{equation}\label{eq:isosK-theory}
\Hom\left(\Uloc(\cA),\, \Uloc(\cB)[-n]\right) \simeq \bbK_n(\rep(\cA, \cB))\,,
\end{equation}
where $\bbK$ denotes non-connective $K$-theory and $\rep(-,-)$ the internal Hom-functor in the homotopy category of dg categories (see~\S\ref{sub:dg}). The composition operation is induced by the tensor product of bimodules. In particular, when $\cA$ is the dg category $\underline{k}$ associated to the base ring $k$ (with one object and $k$ as the dg algebra of endomorphisms) we obtain
\begin{equation}\label{eq:isosK-theory1}
\Hom\left(\Uloc(\underline{k}),\, \Uloc(\cB)[-n]\right) \simeq \bbK_n(\cB)\,.
\end{equation}
At this point it is natural to ask the following motivational questions\,:
\smallbreak
\noindent\textbf{Question~A: }\textit{Which (further) invariants of dg categories can be expressed in terms of morphism sets in the category of non-commutative motives\,?}
\smallbreak
\noindent\textbf{Question~B: }\textit{How to explicitly describe the composition operation in these new cases\,?}
\smallbreak
Intuitively, our answer is ``Bivariant cyclic cohomology, with Connes' bilinear pairings playing the role of the composition operation''.
\subsection{Bivariant cyclic cohomology}
Jones and Kassel, by drawing inspiration from Kasparov's $KK$-theory~\cite{Kasparov}, introduced in \cite{JK} the bivariant cyclic cohomology theory of unital associative $k$-algebras. One of the fundamental properties of this bivariant theory is the fact that it simultaneously extends both negative cyclic homology as well as cyclic cohomology. 
 Bivariant cyclic cohomology $HC^{\ast}(-,-)$ and its two variants, bivariant Hochschild cohomology $HH^*(-,-)$~\cite[\S5.5]{Loday} and bivariant periodic cyclic cohomology $HP^{\ast}(-,-)$~\cite[\S8]{JK}, extend naturally from $k$-algebras to dg categories. They associate to any pair of dg categories $(\cB, \cC)$ a $\bbZ$-graded $k$-module which is contravariant on $\cB$ and covariant on $\cC$.
\newpage
Our answer to the above Question A is the following.
\begin{theoremA}\label{thm:A}
There exist triangulated functors
\begin{equation}\label{eq:triangfunc}
\cT^{H}, \cT^{C}, \cT^{P} : \Mloc(e) \too \Mloc(e)
\end{equation}
such that for any $m \in \bbZ$ and dg categories $\cB$ and $\cC$, we have natural isomorphisms\,:
\begin{eqnarray}
\Hom\left(\Uloc(\cB)[-m], \cT^{H}(\Uloc(\cC))\right)&\simeq& HH^m(\cB, \cC) \label{eq:iso-1}\\
\Hom\left(\Uloc(\cB)[-m], \cT^{C}(\Uloc(\cC))\right)&\simeq& HC^m(\cB, \cC)   \label{eq:iso-2}\\
\Hom\left(\Uloc(\cB)[-m], \cT^{P}(\Uloc(\cC))\right) &\simeq& HP^m(\cB,\cC) \,, \label{eq:iso-3}
\end{eqnarray}
where \eqref{eq:iso-3} holds under the hypothesis that $\Uloc(\cB)$ is compact~\cite{Neeman} in $\Mloc(e)$.
\end{theoremA}
Thanks to \cite[Corollary~7.7]{CT1} if $\cB$ is a saturated dg category in the sense of Kontsevich, then $\Uloc(\cB)$ is compact object in $\Mloc(e)$. Recall from~\cite[\S5.5.1]{Loday}\cite{JK} that we have the following isomorphisms\,:
 \begin{eqnarray}
HH^{\ast}(\underline{k},\cC)\simeq HH_{-\ast}(\cC) && HH^{\ast}(\cB, \underline{k})\simeq HH^{\ast}(\cB) \label{eq:iso-4}\\
 HC^{\ast}(\underline{k}, \cC) \simeq HC^-_{-\ast}(\cC) && HC^{\ast}(\cB, \underline{k}) \simeq HC^{\ast}(\cB) \label{eq:iso-5}\\
HP^{\ast}(\underline{k},\cC) \simeq HP_{-\ast}(\cC)\  && HP^{\ast}(\cB,\underline{k}) \simeq HP^{\ast}(\cB)\,.\label{eq:iso-6}
\end{eqnarray}
Therefore, by replacing $\cB$ or $\cC$ by $\underline{k}$ we obtain two important instantiations of Theorem~A\,: if we replace $\cC$ by $\underline{k}$, the Hochschild $HH^{\ast}(-)$, cyclic $HC^{\ast}(-)$, and periodic cyclic $HP^{\ast}(-)$, cohomology theories become representable in $\Mloc(e)$; on the other hand if we replace $\cB$ by $\underline{k}$, the above isomorphisms \eqref{eq:iso-4}-\eqref{eq:iso-6} combined with isomorphism \eqref{eq:isosK-theory1} show us that the triangulated functors (\ref{eq:triangfunc}) allow us to ``switch'', inside the category of non-commutative motives, from algebraic $K$-theory to the Hochschild, negative cyclic, and periodic cyclic, homology.
\subsection*{Connes' bilinear pairings}
In his foundational work on non-commutative geometry, in the early eighties, Connes~\cite{Connes} discovered bilinear pairings
\begin{equation}\label{eq:pairings}
\langle-,-\rangle: K_0(\cB) \times HC^{2j}(\cB) \too k \qquad j \geq 0
\end{equation}
relating the Grothendieck group with the even part of cyclic cohomology. These bilinear pairings, which were the main motivation behind the development of a cyclic theory, consist roughly on the evaluation of a cyclic cochain at an idempotent representing a finitely generated projective module over $\cB$. 

Now, the above isomorphisms \eqref{eq:isosK-theory}-\eqref{eq:isosK-theory1} and \eqref{eq:iso-1}-\eqref{eq:iso-3} show us that both algebraic $K$-theory as well as the different bivariant cohomology theories can be expressed in terms of morphism sets in the category of non-commutative motives. Therefore, the composition operation in the triangulated category $\Mloc(e)$, combined with these isomorphisms, furnish us bilinear pairings\,:
\begin{eqnarray}
\bbK_n(\rep(\cA,\cB)) \times HH^m(\cB,\cC) &\too & HH^{m-n}(\cA,\cC)  \label{eq:pairing-1}\\
\bbK_n(\rep(\cA,\cB)) \times HC^m(\cB,\cC) &\too & HC^{m-n}(\cA,\cC)  \label{eq:pairing-2}\\
\bbK_n(\rep(\cA,\cB)) \times HP^m(\cB,\cC) &\too & HP^{m-n}(\cA,\cC)  \label{eq:pairing-3}\,.
\end{eqnarray}
Our answer to the above Question B is the following\,:
\begin{theoremB}
The bilinear pairing (\ref{eq:pairing-2}), with $n=0$, $\cA=\cC=\underline{k}$ and $m=2j$, corresponds to Connes' original bilinear pairing (\ref{eq:pairings}). 
\end{theoremB}  
Theorem~B supports the Grothendieckian belief that all classical constructions in (non-commutative) geometry should become conceptually clear in the correct category of (non-commutative) motives. The above pairings \eqref{eq:pairing-1}-\eqref{eq:pairing-3}, which correspond to the composition operation in the category of non-commutative motives, can therefore be considered as an extension of Connes' foundational work.

\section{Background on dg categories and derivators}\label{sec:background}
Throughout this article $k$ denotes a commutative base ring with unit ${\bf 1}$. Adjunctions are displayed vertically with the left (resp. right) adjoint on the left- (resp. right-) hand side. 
\subsection{Dg categories}\label{sub:dg}
Let $\cC(k)$ be the category of (unbounded) complexes of $k$-modules. A {\em differential graded (=dg) category} is a category enriched over $\cC(k)$ and a {\em dg functor} is a functor enriched over $\cC(k)$; consult Keller's survey~\cite{ICM}. The category of dg categories will be denoted by $\dgcat$.

Given dg categories $\cA$ and $\cB$ their {\em tensor product}
$\cA \otimes \cB$ is defined as follows: the set of objects is the cartesian product and given objects $(x,z)$ and $(y,w)$ in $\cA \otimes \cB$, we set $(\cA \otimes \cB)((x,z),(y,w)):= \cA(x,y) \otimes \cB(z,w)$.

A dg functor $F: \cA \to \cB$ is a called a {\em derived Morita equivalence}
if it induces an equivalence $\cD(\cA) \stackrel{\sim}{\to}~\cD(\cB)$ between the associated derived categories. Thanks to \cite[Theorem~5.3]{IMRN} the category $\dgcat$ carries a (cofibrantly generated) Quillen model structure~\cite{Quillen} whose weak equivalences are the derived Morita equivalences. We denote by $\Hmo$ the homotopy category hence obtained.

The tensor product of dg categories can be derived into a bifunctor $-\otimes^\bbL-$ on $\Hmo$. Moreover, this bifunctor admits an internal Hom-functor $\rep(-,-)$. Given dg categories $\cA$ and $\cB$, $\rep(\cA,\cB)$ is the full dg subcategory of $\cA\text{-}\cB$-bimodules $X$ such that, for every object $x$ in $\cA$, the right $\cB$-module $X(-,x)$ is a compact object in the triangulated category $\cD(\cB)$; see \cite[\S2.4]{CT1}.

\subsection{Grothendieck derivators}\label{sec:Grothendieck}
Derivators allow us to state and prove precise universal properties and to dispense with many of the technical problems one faces in using Quillen model categories; consult Grothendieck's original manuscript~\cite{Grothendieck} or \cite[\S1]{CN} for a short account. Given a Quillen model category $\cM$, we will denote by $\HO(\cM)$ its associated derivator. In order to simplify the exposition, a morphism of derivators and its value at the base category $e$ (which has one object and one morphism) will be denoted by the same symbol. It will be clear from the context which situation we are referring to.

\section{Proof of Theorem~A}\label{sec:proofA}
We start by proving that the category of non-commutative motives satisfies an important compactness property.
\begin{proposition}\label{prop:well-gen}
The triangulated category $\Mloc(e)$ of non-commutative motives is well-generated in the sense of Neeman~\cite[Remark~8.1.7]{Neeman}.
\end{proposition}
\begin{proof}
Recall from \cite[\S7.1]{CT1} that the triangulated category $\Mloc(e)$ can be realized as the homotopy category of a stable Quillen model category $\cM ot_{\dg}^{\loc}$. This model category is obtained by taking left Bousfield localizations of presheaves of symmetric spectra on a small category $\dgcat_{\mathsf{f}}$; see the proof of \cite[Theorem~7.5]{CT1}. Since symmetric spectra~\cite{HSS} is a combinatorial model category and this class of model categories is stable under left Bousfield localizations and passage to presheaves, the model category $\cM ot_{\dg}^{\loc}$ is combinatorial in the sense of Smith. Therefore, thanks to \cite[Proposition~6.10]{Rosicky} we conclude that  $\Mloc(e)$ is well-generated in the sense of Neeman. 
\end{proof}
Now, recall from \cite[Examples~7.9 and 7.10]{CT1} the construction of the Hochschild homology and mixed complex localizing invariants
\begin{eqnarray*}
HH: \HO(\dgcat) \too \HO(\cC(k)) &  & C: \HO(\dgcat) \too \HO(\cC(\Lambda))\,.
\end{eqnarray*} 
Here, $\Lambda$ stands for the dg algebra $k[\epsilon]/\epsilon^2$, where the variable $\epsilon$ is of degree $-1$ and satisfies $d(\epsilon)=~0$. Thanks to the equivalence of categories~\eqref{eq:cat}, these morphisms of derivators factor (uniquely) through the universal localizing invariant $\Uloc$. We obtain then triangulated functors
\begin{eqnarray}\label{eq:triang-func}
HH_{\loc}: \Mloc(e) \too \cD(k) &  & C_{\loc}: \Mloc(e) \too \cD(\Lambda)\,,
\end{eqnarray} 
and natural equivalences
\begin{eqnarray}\label{eq:triang-eq}
HH \simeq HH_{\loc} \circ \Uloc &  & C \simeq C_{\loc} \circ \Uloc \,.
\end{eqnarray} 
\begin{proposition}\label{prop:adjs}
The triangulated functors \eqref{eq:triang-func} admit right adjoints.
\end{proposition}
\begin{proof}
The proof will consist on verifying the conditions of \cite[Theorem~8.4.4]{Neeman}. First, observe that the triangulated categories $\Mloc(e)$, $\cD(k)$ and $\cD(\Lambda)$ have small Hom-sets since they can be realized as homotopy categories of (stable) Quillen model categories. Thanks to Proposition~\ref{prop:well-gen} the category $\Mloc(e)$ is well-generated in the sense of Neeman and so by \cite[Theorem~1.17]{Neeman} it satisfies the representability theorem. Finally, since the left-hand side of the equivalence of categories \eqref{eq:cat} consists of homotopy colimit preserving morphisms of derivators, we conclude that the triangulated functors \eqref{eq:triang-func} respect arbitrary coproducts. This achieves the proof.
\end{proof}
Let us denote by  
\begin{eqnarray}\label{eq:adjs}
\xymatrix{
\cD(k)\ar@<1ex>[d]^{R_{H}} \\
\Mloc(e) \ar@<1ex>[u]^{HH_{\loc}}
}
&&
\xymatrix{
\cD(\Lambda)\ar@<1ex>[d]^{R_C} \\
\Mloc(e) \ar@<1ex>[u]^{C_{\loc}}
}
\end{eqnarray}
the adjunctions of triangulated categories given by Proposition~\ref{prop:adjs}. The triangulated endofunctors $R_{H}\circ HH_{\loc}$ and $R_C \circ C_{\loc}$ of $\Mloc(e)$ will be denoted, respectively, by $\cT^{H}$ and $\cT^{C}$. 

For any $m \in \bbZ$ and dg categories $\cB$ and $\cC$, we have natural isomorphisms\,:
\begin{eqnarray}
\Hom \left(\Uloc(\cB)[-m], \cT^{H}(\Uloc(\cC))\right) & \simeq& \Hom_{\cD(k)}(HH(\cB)[-m], HH(\cC)) \label{eq:equiv-1}\\
&\simeq  & HH^m(\cB,\cC)  \label{eq:equiv-2}\,.
\end{eqnarray}
Isomorphism \eqref{eq:equiv-1} follows from the left-hand side adjunction in \eqref{eq:adjs} and from the left-hand side equivalence in \eqref{eq:triang-eq}. Isomorphism \eqref{eq:equiv-2} follows from \cite[Definition~5.5.5.1]{Loday}\footnote{In {\it loc.~cit.} the author used the symbol $C$, instead of $HH$, to denote the Hochschild complex.} and from the fact that homology in degree $-m$ of the complex of maps between $HH(\cB)$ and $HH(\cC)$ is naturally isomorphic to the set of maps in $\cD(k)$ between $HH(\cB)[-m]$ and $HH(\cC)$. This proves isomorphism \eqref{eq:iso-1} in Theorem~A.

In the mixed complex case we have similar natural isomorphisms\,:
\begin{eqnarray}
\Hom \left(\Uloc(\cB)[-m], \cT^{C}(\Uloc(\cC))\right) & \simeq&  \Hom_{\cD(\Lambda)}(C(\cB)[-m], C(\cC)) \label{eq:equiva-3}\\
&\simeq  & \mathrm{Ext}^m_\Lambda(C(\cB), C(\cC))   \label{eq:equiva-4}\\
&=  & HC^m(\cB,\cC)  \,. \label{eq:equiva-5}
\end{eqnarray}
Isomorphism \eqref{eq:equiva-3} follow from the righ-hand side adjunction in \eqref{eq:adjs} and from the right-hand side equivalence in \eqref{eq:triang-eq}. Isomorphism \eqref{eq:equiva-4} is a standard fact in homological algebra. Equality \eqref{eq:equiva-5} follows from the definition of bivariant cyclic cohomology; see \cite[page 2]{JK}. This proves isomorphism \eqref{eq:iso-2} in Theorem~A.

Let us now prove isomorphism \eqref{eq:iso-3} in Theorem~A. Given any object $M$ in $\cD(\Lambda)$, we have a functorial periodicity map $S$ in $\mathrm{Ext}_\Lambda^2(M,M)\simeq \Hom_{\cD(\Lambda)}(M,M[2])$;  see~\cite[\S1]{JK}. This gives rise to an induced natural transformation of triangulated functors
\begin{equation*}
S: \cT^{C}=(R_C \circ C_{\loc}) \Rightarrow (R_C \circ (-[2]) \circ C_{\loc})=: \cT^{C}[2]\,.
\end{equation*}
Let $\cT^{P}$ be the homotopy colimit~\cite[Definition~1.6.4]{Neeman} of the following diagram of natural transformations of triangulated functors
$$ \cT^{C} \stackrel{S}{\Rightarrow} \cT^{C}[2] \stackrel{S}{\Rightarrow} \cdots \stackrel{S}{\Rightarrow} \cT^{C}[2r] \stackrel{S}{\Rightarrow} \cdots\,.$$
For any $m \in \bbZ$ and dg categories $\cB$ and $\cC$, such that $\Uloc(\cB)$ is compact in $\Mloc(e)$, we have  natural isomorphisms\,:
\begin{eqnarray}
\Hom \left(\Uloc(\cB)[-m], \cT^{P}(\Uloc(\cC))\right) & =& \Hom \left(\Uloc(\cB)[-m], \varinjlim_r\, \cT^{C}[2r](\Uloc(\cC))\right)\nonumber\\
&\simeq  &\varinjlim_r\, \Hom\left(\Uloc(\cB)[-m], \cT^{C}[2r](\Uloc(\cC))\right)  \label{eq:equiv-4} \\
&\simeq  &\varinjlim_r\, \Hom\left(\Uloc(\cB)[-m], \cT^{C}(\Uloc(\cC))[2r]\right)  \label{eq:equiv-5} \\
&\simeq  &\varinjlim_r\, \Hom\left(\Uloc(\cB)[-m-2r], \cT^{C}(\Uloc(\cC))\right)  \nonumber \\
&\simeq  &\varinjlim_r\, HC^{m+2r}(\cB,\cC)  \label{eq:equiv-6} \\
&=  & HP^m(\cB,\cC)  \label{eq:equiv-7} \,.
\end{eqnarray}
Isomorphism (\ref{eq:equiv-4}) follows from the compactness of $\Uloc(\cB)$ in the triangulated category $\Mloc(e)$. Isomorphism (\ref{eq:equiv-5}) follows from the natural equivalence of triangulated functors
$$ \cT^{C}[2r]:=(R_C \circ (-[2r])\circ C_{\loc}) \simeq (R_C \circ C_{\loc})[2r]\,.$$ Isomorphism (\ref{eq:equiv-6}) follows from the isomorphism (\ref{eq:iso-2}) of Theorem~A, which was already proven. Finally, equality (\ref{eq:equiv-7}) follows from the definition of bivariant periodic cyclic cohomology; see \cite[Definition~8.1]{JK}.
\section{Proof of Theorem~B}
We start by describing the classical Chern character map in terms of the category of non-commutative motives. Recall from isomorphism \eqref{eq:isosK-theory1} that, given any dg category $\cB$, we have a natural isomorphism
$$ K_0(\cB)\simeq \Hom(\Uloc(\underline{k}), \Uloc(\cB))\,.$$
By combining isomorphism \eqref{eq:iso-5} with isomorphisms \eqref{eq:equiva-4}-\eqref{eq:equiva-5}, we obtain also the following identification
$$ \Hom_{\cD(\Lambda)}(C(\underline{k}),C(\cB)) \simeq HC^-_0(\cB)\,.$$
Therefore, since the triangulated functor
$$ C_{\loc}: \Mloc(e) \too \cD(\Lambda)$$
sends $\Uloc(\underline{k})$ to $C(\underline{k})$, we obtain an induced map
\begin{equation}\label{eq:Chern}
K_0(\cB) \simeq \Hom(\Uloc(\underline{k}), \Uloc(\cB)) \too \Hom_{\cD(\Lambda)}(C(\underline{k}), C(\cB)) \simeq HC^-_0(\cB)\,.
\end{equation}
\begin{proposition}\label{prop:Chern}
Given any dg category $\cB$, the above map \eqref{eq:Chern} is the Chern character map $ch^-(\cB)$ of $\cB$; see \cite[\S8.3]{Loday}.
\end{proposition}
\begin{proof}
Thanks to isomorphism \eqref{eq:isosK-theory1} the Grothendieck group functor can be expressed as the following composition
\begin{equation}\label{eq:Groth}
K_0: \dgcat \too \Hmo \stackrel{\Uloc}{\too} \Mloc(e) \stackrel{\Hom(\Uloc(\underline{k}),-)}{\too} \mathsf{Ab}\,,
\end{equation}
where $\mathsf{Ab}$ denotes the category of abelian groups. Similarly, isomorphism \eqref{eq:iso-4} combined with isomorphisms \eqref{eq:equiva-4}-\eqref{eq:equiva-5}, allow us to express the degree zero negative cyclic homology group functor as the following composition
\begin{equation}\label{eq:negC}
HC^-_0: \dgcat \too \Hmo \stackrel{C}{\too} \cD(\Lambda) \stackrel{\Hom_{\cD(\Lambda)}(C(\underline{k}),-)}{\too} \mathsf{Ab}\,.
\end{equation}
Since we have a commutative diagram
$$
\xymatrix{
\Hmo \ar[r]^C \ar[d]_{\Uloc} & \cD(\Lambda) \\
\Mloc(e) \ar[ur]_{C_{\loc}} & \,,
}
$$
the triangulated functor $C_{\loc}$ gives rise to a natural transformation $K_0 \Rightarrow HC_0^-$ between the above functors \eqref{eq:Groth}-\eqref{eq:negC}, whose value at an arbitrary dg category $\cB$ is the above map \eqref{eq:Chern}. 

Now, recall from \cite[Theorem~1.3]{Chern} that there is an isomorphism
\begin{eqnarray}\label{eq:nat}
\mathrm{Nat}(K_0, HC^-_0) \stackrel{\sim}{\too} HC^-_0(\underline{k})\simeq k & & ch^- \mapsto {\bf 1}
\end{eqnarray}
between natural transformations and the base ring $k$. This isomorphism is given by the evaluation of a natural transformation at the class $[k]$ of $k$ (as a module over itself) in the Grothendieck group $K_0(\underline{k})=K_0(k)$. Moreover, under this isomorphism the Chern character map $ch^-$ corresponds to the unit ${\bf 1}$ of the base ring $k$. The class $[k]$ of $k$ in $K_0(\underline{k})$ is given by the identity map in $\Hom(\Uloc(\underline{k}), \Uloc(\underline{k}))$ and the identity map in $\Hom_{\cD(\Lambda)}(C(\underline{k}), C(\underline{k}))$ corresponds, under the natural isomorphisms
$$ \Hom_{\cD(\Lambda)}(C(\underline{k}),C(\underline{k})) \simeq \mathrm{Ext}^0_{\Lambda}(C(\underline{k}),C(\underline{k})) = HC^0(\underline{k},\underline{k})\simeq HC^-_0(\underline{k})\simeq k\,,$$
to the identity ${\bf 1}$ of $k$; see \cite[Theorem~2.3]{JK}. Therefore, the above isomorphism \eqref{eq:nat} allow us to conclude that given any dg category $\cB$, the map \eqref{eq:Chern} is in fact the Chern character map $ch^-(\cB)$ of $\cB$. This achieves the proof.
\end{proof}
Now, recall from \cite[\S8.3.10]{Loday} that Connes' bilinear pairings (\ref{eq:pairings}) can be expressed as the following compositions
\begin{equation}\label{eq:composition}
\langle-,-\rangle : K_0(\cB) \times HC^{2j}(\cB) \stackrel{ch_{2j} \times \id}{\too} HC_{2j}(\cB) \times HC^{2j}(\cB) \stackrel{\mathrm{ev}}{\too} k\qquad j\geq 0\,,
\end{equation}
where $ch_{2j}$ is the Chern character map (see \cite[\S8.3]{Loday}) and $\mathrm{ev}$ is induced by the evaluation of cyclic cochains on cyclic chains. Thanks to the adjunction
$$
\xymatrix{
\cD(\Lambda)\ar@<1ex>[d]^{R_C} \\
\Mloc(e) \ar@<1ex>[u]^{C_{\loc}} 
}
$$
we obtain a commutative square\footnote{In order to reduce the size of the diagram we have denoted, in the upper row, the universal localizing invariant $\Uloc$ by $\cU$.}
$$
\xymatrix{
\Hom(\cU(\underline{k}), \cU(\cB)) \times \Hom(\cU(\cB), \cT^{C}(\cU(\underline{k}))[2j]) \ar[r]^-{\mathrm{comp}} \ar[d]_{C_{\loc}\times \varphi} & \Hom(\cU(\underline{k}), \cT^{C}(\cU(\underline{k}))[2j]) \ar[d]^{\varphi}_{\sim}\\
\Hom_{\cD(\Lambda)}(C(\underline{k}), C(\cB)) \times \Hom_{\cD(\Lambda)}(C(\cB), C(\underline{k})[2j]) \ar[r]_-{\mathrm{comp}} & \Hom_{\cD(\Lambda)}(C(\underline{k}), C(\underline{k})[2j])\,,
}
$$
where $\varphi$ is the natural isomorphism given by the adjunction and the horizontal maps are the composition operations in $\cD(\Lambda)$ and $\Mloc(e)$. Proposition~\ref{prop:Chern}, combined with isomorphisms \eqref{eq:iso-2}, \eqref{eq:iso-5} and \eqref{eq:equiva-4}-\eqref{eq:equiva-5}, show us that the above commutative square correspond to the following diagram
\begin{equation}\label{eq:com1}
\xymatrix{
K_0(\cB) \times HC^{2j}(\cB) \ar[rr] \ar[d]_{ch^-(\cB)\times \id} && HC^{2j}(\underline{k})\simeq k \ar@{=}[d] \\
HC_0^-(\cB) \times HC^{2j}(\cB) \ar[rr]_-{\mathrm{comp}} && HC^{2j}(\underline{k})\simeq k\,,
}
\end{equation}
where the upper horizontal map is the pairing \eqref{eq:pairing-2} (with $n=0$, $\cA=\cC=\underline{k}$ and $m=2j$). Now, recall from \cite[\S5.1.8]{Loday} that there exist natural maps 
$$U_j: HC^-_0(\cB) \too HC_{2j}(\cB)\qquad j \geq 0$$
such that $U_j \circ ch^-(\cB) = ch_{2j}(\cB)$. Thanks to the description of the composition operation in $\cD(\Lambda)$ given in \cite[Theorem~5.1]{JK} we have the following diagram
\begin{equation}\label{eq:com2}
\xymatrix{
HC^-_0(\cB) \times HC^{2j}(\cB) \ar[rr]^-{\mathrm{comp}} \ar[d]_{U_j \times \id} && k \ar@{=}[d] \\
HC_{2j}(\cB) \times HC^{2j}(\cB) \ar[rr]_-{\mathrm{ev}} && k\,.
}
\end{equation}
Finally, by combining diagram (\ref{eq:com1}) with diagram (\ref{eq:com2}) we conclude that the pairing \eqref{eq:pairing-2} (with $n =0$, $\cA=\cC=\underline{k}$ and $m=2j$) identifies with the above composition \eqref{eq:composition}, and so with Connes' original bilinear pairing \eqref{eq:pairings}. This achieves the proof.

\section{An application\,: (de)suspension of bivariant cohomology theories}
In \cite{suspension} the author constructed a simple model for the suspension in the triangulated category of non-commutative motives. Consider the $k$-algebra $\Gamma$ of $\bbN\times \bbN$-matrices $A$ which satisfy the following two conditions\,: the set $\{A_{i,j}\, |\, i, j \in \bbN \}$ is finite and there exists a natural number $n_A$ such that each row and each column has at most $n_A$ non-zero entries. Let $\Sigma$ be the quotient of $\Gamma$ by the two-sided ideal consisting of those matrices with finitely many non-zero entries; see~\cite[\S3]{suspension}. Alternatively, consider the (left) localization of $\Gamma$ with respect to the matrices $\overline{I_n}, n \geq 0$, with entries $(\overline{I_n})_{i,j}={\bf 1}$ for $i=j >n$ and $0$ otherwise. Then, for any dg category $\cA$ we have a canonical isomorphism
\begin{equation}\label{eq:susp}
\Uloc(\Sigma(\cA)) \stackrel{\sim}{\too} \Uloc(\cA)[1]
\end{equation}
in $\Mloc(e)$, where $\Sigma(\cA)=\cA\otimes \Sigma$. Note that \eqref{eq:susp} shows us that if $\Uloc(\cB)$ is a compact object in $\Mloc(e)$ then so it is $\Uloc(\Sigma(\cB))$.
Now, by combining isomorphism~\eqref{eq:susp} with Theorem~A we obtain the following result.
\begin{theorem}\label{thm:susp}
For any dg categories $\cB$ and $\cC$, we have natural isomorphisms
\begin{eqnarray}
HH^{\ast+1}(\Sigma(\cB),\cC) \simeq HH^{\ast}(\cB,\cC) & & HH^{\ast-1}(\cB,\Sigma(\cC)) \simeq HH^{\ast}(\cB,\cC) \label{eq:new1} \\
HC^{\ast+1}(\Sigma(\cB),\cC) \simeq HC^{\ast}(\cB,\cC)& & HC^{\ast-1}(\cB,\Sigma(\cC)) \simeq HC^{\ast}(\cB,\cC) \label{eq:new2} \\
HP^{\ast+1}(\Sigma(\cB),\cC) \simeq HP^{\ast}(\cB,\cC)& & HP^{\ast-1}(\cB,\Sigma(\cC)) \simeq HP^{\ast}(\cB,\cC) \,, \label{eq:new3}
\end{eqnarray}
where the left-hand side of \eqref{eq:new3} holds under the hypothesis that $\Uloc(\cB)$ is compact (for instance if $\cB$ is a saturated dg category in the sense of Kontsevich). 
\end{theorem}
Theorem~\ref{thm:susp} extends Kassel's previous work~\cite[\S III Theorem~3.1]{Kassel} on bivariant cyclic cohomology on ordinary algebras defined over a field to dg categories defined over a general commutative base ring. Hence, it can now be applied to schemes. Given a (quasi-compact and quasi-separated) $k$-scheme $X$, it is well-known that the category of perfect complexes of $\cO_X$-modules admits a dg-enhancement $\perf_\dg(X)$; see for instance \cite{LO} or ~\cite[Example~4.5]{CT1}. Moreover, whenever $X$ is smooth and proper the dg category $\perf_\dg(X)$ is saturated in the sense of Kontsevich. Given a pair $(X,Y)$ of $k$-schemes, the bivariant Hochschild, cyclic, and periodic, homology of $(X,Y)$ can be obtained from the pair of dg categories $(\perf_\dg(X),\perf_\dg(Y))$ by applying the corresponding bivariant theory. Therefore, when $\cB=\perf_\dg(X)$ and $\cC=\perf_\dg(Y)$, the above isomorphisms \eqref{eq:new1}-\eqref{eq:new3} reduce to the corresponding isomorphisms associated to the schemes $X$ and $Y$.

Note also that the isomorphisms \eqref{eq:iso-4}-\eqref{eq:iso-6} lead to important specializations of Theorem~\ref{thm:susp} when we replace $\cB$ or $\cC$ by $\underline{k}$. When $\cB=\underline{k}$, $\Sigma(\cC)$ corresponds to the suspension of $\cC$ in the different homology theories. When $\cC=\underline{k}$, $\Sigma(\cB)$ corresponds to the desuspension.

\subsection*{Acknowledgments:} The author is very grateful to Clark Barwick, Sasha Beilinson, Dan Grayson, Haynes Miller, and Jack Morava for stimulating discussions. He would also like to thank the Departments of Mathematics at MIT and UIC for their hospitality (where this work was carried out) and the Midwest Topology Network for financial support.

\end{document}